\documentclass[11pt]{amsart}
\usepackage{geometry, tikz-cd}                
\usepackage{amsthm}
\usepackage{graphicx}
\usepackage{amssymb,amsmath}
\usepackage{epstopdf}
\usepackage{xcolor}
\usepackage{ulem}
\usepackage{cancel}
\usepackage[percent]{overpic}
\usepackage{stackengine}

\numberwithin{equation}{section}

\theoremstyle{plain}
\newtheorem{thm}{Theorem}[section]

\newtheorem{prop}[thm]{Proposition}
\newtheorem{lem}[thm]{Lemma}

\newtheorem{rem}[thm]{Remark}

\makeatletter
\@namedef{subjclassname@2020}{%
  \textup{2020} Mathematics Subject Classification}
\makeatother

\makeatletter
\def\moverlay{\mathpalette\mov@rlay}
\def\mov@rlay#1#2{\leavevmode\vtop{%
   \baselineskip\z@skip \lineskiplimit-\maxdimen
   \ialign{\hfil$\m@th#1##$\hfil\cr#2\crcr}}}
\newcommand{\charfusion}[3][\mathord]{
    #1{\ifx#1\mathop\vphantom{#2}\fi
        \mathpalette\mov@rlay{#2\cr#3}
      }
    \ifx#1\mathop\expandafter\displaylimits\fi}
\makeatother

\title{Counting reciprocal hyperbolic elements   in Hecke groups}

\author{Ara Basmajian}
\address[Ara Basmajian]{The Graduate Center, CUNY \\ 365 Fifth Ave., N.Y., N.Y., 10016 and Hunter College, CUNY \\ 695 Park Ave., N.Y., N.Y., 10065, USA}
\email{abasmajian@gc.cuny.edu}
 \thanks{The first two authors dedicate this work  to the third author, Robert Suzzi Valli,  who passed away before the completion of this manuscript.}
 \thanks{A.B.  partially support by a PSC-CUNY Grant and a Simons Collaboration Grant (359956, A.B.).}

\author{Blanca Marmolejo}
\address[Blanca Marmolejo]{Bluetab, an IBM Company, 28020 Madrid, Spain}
\email{bmarmolejo1@gmail.com}
\thanks{}

 \author{Robert Suzzi Valli}

\email{}
\thanks{}

\begin{document}
\begin{abstract}  A {\it reciprocal geodesic}  on a $(2,k, \infty)$ Hecke surface is  a geodesic loop based  at an even order cone point $p$  traversing its path an even number of times. Associated to each reciprocal geodesic is the conjugacy class of a  hyperbolic element in the 
$(2,k,\infty)$ Hecke group whose axis  passes through a cone point that projects to $p$. Such an element is called a 
{\it reciprocal hyperbolic element}  based at $p$.

 In this paper,  we  determine the asymptotic growth rate and limiting constant 
(in terms of word length) of the number of  primitive conjugacy classes of  reciprocal  hyperbolic elements in  a Hecke group.   
\end{abstract}
\maketitle


\section{Introduction}

Our goal in this paper  is to find the  growth rate and the asymptotic constants (in terms of word length) of  reciprocal geodesics as well as the primitive reciprocal geodesics on  Hecke surfaces.  For $k \geq 3$ an integer, the 
{\it $(2,k, \infty)$  Hecke group} 
is  the image of the  discrete faithful representation  
$\mathbb{Z}_2 \ast \mathbb{Z}_k \rightarrow \textrm{PSL}(2,\mathbb{R})$, given by 

$$
a \mapsto A:=\begin{pmatrix} 
     0 & -1   \\
     1 & 0 
\end{pmatrix} \text{ and  } b \mapsto B :=\begin{pmatrix} 
     0 & -1    \\
     1 & 2\cos \frac{\pi}{k}
\end{pmatrix}
$$
where $a$ 
is  the generator of $\mathbb{Z}_{2}$  and $b$ is the generator of  
$\mathbb{Z}_k$.  We use the convention that $k=\infty$ corresponds to the group  $\mathbb{Z}_2 \ast \mathbb{Z}$, and hence  the   Hecke group representation  for 
$\mathbb{Z}_2 \ast \mathbb{Z}  \rightarrow \textrm{PSL}(2,\mathbb{R})$, is given by 

$$
a \mapsto A=\begin{pmatrix} 
     0 & -1   \\
     1 & 0 
\end{pmatrix} \text{ and  } b \mapsto B:=\begin{pmatrix} 
     0 & -1    \\
     1 & 2
\end{pmatrix}
$$
We note that $BA$ is parabolic for all $k=3,4,...,\infty$. Moreover for $k=\infty$, $B$ is parabolic as well. 

Now consider the {\it Hecke surface};  that is, the $(2,k,\infty)$ triangle orbifold,  $X_{k}=\Bbb{H}/ \Gamma_{k}$, where $\Gamma_{k}=<A,B>$. If  $k=\infty$, we use $X_{\infty}$ to denote  the $(2,\infty,\infty)$ orbifold.  Geometrically the  orbifold  $X_{k}$ can be realized by gluing two isometric copies of the 
$(\frac{\pi}{2}, \frac{\pi}{k}, 0)$-hyperbolic  triangle along  common sides as in Figure \ref{fig:recgeod}. Throughout this work, by a closed geodesic on $X_{k}$ we mean the projection of the axis of a hyperbolic element in $\Gamma_{k}$.

\begin{figure}[t]
\begin{center}
\begin{overpic}[scale=.5]{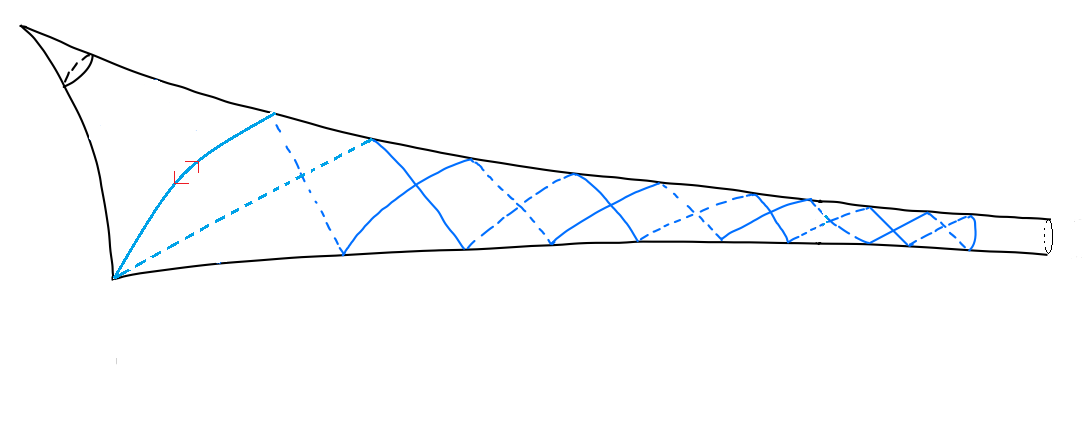}
\put(0,36.5){\footnotesize{$k$}}
\put(10,11){\footnotesize{$2$}}
\put(14,24){$\gamma$}
\end{overpic}
\vspace*{-.75in}
\end{center}
\caption{A reciprocal geodesic $\gamma$ based at order 2 on $X_k$}
\label{fig:recgeod}
\end{figure}

Let $p$ be an even order cone point of the Hecke surface $X_{k}$. A {\it reciprocal geodesic based at $p \in X_{k}$}   is a  geodesic that starts and ends at $p$, traversing its image an even number of times. Equivalently,  such a geodesic    lifts to a  hyperbolic element  that is the product of two  distinct conjugate involutions whose fixed  points project to $p$.
Such a hyperbolic element we call a {\it reciprocal hyperbolic element at $p$}.  
Since we are exclusively working with $(2,k,\infty)$ surfaces, our reciprocal geodesics are either based at the order  $k$ (if $k$ is even) or  the order 2 cone points.  The  reciprocal geodesics based at the order $k=2m$  cone point corresponds to the involution conjugacy class   $b^{m}$,  and the reciprocal geodesics based at the order 2 cone point correspond to the   involution conjugacy class of the  element $a$.  Of course, in the case that $k$ is odd there is only one conjugacy class of order two involution,   and so reciprocal geodesics arise as hyperbolic elements whose axes pass through the order two cone point corresponding to the element $A \in \Gamma_{k}$. However, in the  $k=2m$ even case there are two involution conjugacy classes: one  corresponding to the involution $A \in  \Gamma_{k}$,  and the other conjugacy class corresponding to the involution $B^{m} \in \Gamma_{k}$.

Our interest is in counting conjugacy classes of reciprocal  hyperbolic elements. One complicating feature that arises   is that there are reciprocal geodesics that are simultaneously based at the order 2 cone point and the order 
$k=2m$ cone point (see Remark \ref{rem: simultaneously in both}).  Nevertheless, we have

\vskip5pt

\noindent{\bf Theorem A.} (Based at order 2)
{\it Let $X_{k}$ be the $(2,k,\infty)$ Hecke surface. Then
\begin{enumerate}

\item  For $k=2m+1$ and $m \geq 2$
\begin{displaymath}
\frac{\big|\{\gamma \subset X_{k} \text{ a reciprocal geodesic  based at order 2}: |\gamma|=2t\}\big|}
{\alpha_{2m+1}^{t}} \xrightarrow{t\to\infty} \frac{c_{2m+1}}{2}
\end{displaymath}
where $c_{2m+1}$ is a positive constant and $\alpha_{2m+1}$ is the unique positive root of the polynomial,

\begin{equation*}
x^{m+1} -2x^{m-1}-2x^{m-2} - \dots -2x-2
\end{equation*}

\item For $k=2m$ and $m \geq 2$
\begin{displaymath}
\frac{\big|\{\gamma \subset X_{k} \text{ a reciprocal geodesic based at order 2}: |\gamma|=2t\}\big|}
{\beta_{2m}^{t}} \xrightarrow{t\to\infty} \frac{d_{2m}}{2}
\end{displaymath}
where $d_{2m}$ is a positive constant and $\beta_{2m}$ is the unique positive root of the polynomial,

\begin{equation*}
 x^{m+1} -2x^{m-1} -2x^{m-2}- \dots -2x-1 
\end{equation*}

\item For $k=\infty$, 
\begin{displaymath}
\frac{\big|\{\gamma \subset X_{\infty} \text{ a reciprocal geodesic  based at order 2} : |\gamma|=2t\}\big|}
{2^{t}} \xrightarrow{t\to\infty} \frac{1}{6}
\end{displaymath}
\end{enumerate}
The same asymptotics hold for the primitive reciprocal geodesics in items (1)-(3).}

\vskip5pt

We next consider reciprocal geodesics based at the order $k$ cone point for $k$ even. 

\vskip5pt

\noindent{\bf Theorem B.} (Based at order $k$)
{\it Fix an integer $m \geq2$ and suppose   
$X_{k}$ is the $(2,k,\infty)$ Hecke surface with $k=2m$  even. Then
\begin{displaymath}
\frac{\big|\{\gamma \subset X_{k} \text{ a reciprocal geodesic based at order 
$k$}: |\gamma|=2t\}\big|}
{\beta_{2m}^{t}} \xrightarrow{t\to\infty} \frac{e_{2m}}{2}
\end{displaymath}
where $e_{2m}$ is a positive constant and $\beta_{2m}$ is the unique positive root of the polynomial,

\begin{equation*}
  x^{m+1} -2x^{m-1} -2x^{m-2}- \dots -2x-1   
\end{equation*}
The same asymptotic statements hold for the primitive reciprocal geodesics on $X_{k}$ for $k=3,4,...,\infty$.}

Table \ref{ta: poly} contains some sample    associated polynomials, dominant roots, and asymptotic constants that appear in   Theorems  $\text{A}$  and $\text{B}$.
\vskip5pt

\noindent{\bf Remark.}
{\it Note that the orbifold surface $X_{\infty}$ is the $3$-fold cover of the modular surface $X_3$, and hence not surprisingly the reciprocal geodesics on $X_{\infty}$ grow at $1/3$ the rate that the reciprocal geodesics on $X_3$ grow.}

\vskip5pt

\noindent{\bf Remark.}
{\it It is elementary to show that   the dominant roots   $\{\alpha_{2m+1}\}_{m=2}^{\infty}$  of the polynomial in (\ref{eq: polynomial and roots})  form an 
 increasing sequence, satisfy $\sqrt{2}< \alpha_{k}< 2$, and converge to $2$. The same holds  for  the dominant roots $\{\beta_{2m}\}_{m=2}^{\infty}$. The positive constants 
 $\{c_{2m+1}\}_{m=1}^{\infty}$,
$\{d_{2m}\}_{m=2}^{\infty}$, and $\{e_{2m}\}_{m=2}^{\infty}$ are each determined by solving a linear system of equations. 
 For more details  on these  properties, and for solving the linear system of equations  to obtain the positive constants  see section \ref{sec: linear system}.}

 When $k$ is even, as a consequence of Theorems 
 $\text{A}$ and $\text{B}$,
 the  reciprocal geodesics based at the cone point of order 2 grow at the same rate as the ones based at the  order $k$.
 More precisely we have
 
\vskip5pt
\noindent{\bf Corollary C.}
{\it Suppose $X_{k}$ is the $(2,k,\infty)$ Hecke surface with 
$k=2m$  even.Then
\begin{equation}
\frac{\big|\{\gamma \subset X_{k} \text{ a reciprocal geodesic based at order 
$k$}: |\gamma|=2t\}\big|}{\big|\{\gamma \subset X_{k}\text{ a reciprocal geodesic based at order 
2}: |\gamma|=2t\}\big|} \xrightarrow{t\to\infty}  \frac{e_{2m}}{d_{2m}}.
\end{equation}
%
In particular, 
\begin{equation}
\frac{\big|\{\gamma \text{ any reciprocal geodesic in $X_{k}$}: |\gamma|=2t\}\big|}
{\beta_{2m}^{t}} \xrightarrow{t\to\infty} \frac{d_{2m}+e_{2m}}{2}
\end{equation}}


\begin{table}\label{ta: poly}
\begin{center}
\begin{tabular}{ |c|c|c|c| } 
 \hline
 Hecke group & Polynomial & Dominant root & Coefficient   \\ 
\hline
\hline
 $\mathbb{Z}_2\ast \mathbb{Z}_3$&$x^{2}-2$ & No dominant root& \\ 
\hline
 $\mathbb{Z}_2\ast \mathbb{Z}_4$&  $x^{3}-2x-1$& $\beta_{4}=1.61803$  & $d_{4}=0.44721$  \\ 
\hline
 $\mathbb{Z}_2\ast \mathbb{Z}_5$&$x^{3}-2x-2$&$\alpha_{5}=1.76929$ & $c_{5}=0.42353$  \\
\hline
 $\mathbb{Z}_2\ast \mathbb{Z}_6$&$x^{4}-2x^{2}-2x-1$ &$\beta_{6}=1.83929$  & $d_{6}=0.40061$  \\
\hline
 $\mathbb{Z}_2\ast \mathbb{Z}_7$&$x^{4}-2x^{2}-2x-2$& $\alpha_{7}=1.89932$ & $c_{7}=0.38472$  \\
\hline
 $\mathbb{Z}_2\ast \mathbb{Z}_8$&$x^{5}-2x^{3}-2x^{2}-2x-1$ &$\beta_{8}=1.92756$ & $d_{8}=0.37289$  \\
\hline
 $\mathbb{Z}_2\ast \mathbb{Z}_9$&$x^{5}-2x^{3}-2x^{2}-2x-2$ & $\alpha_9=1.95350$  & $c_{9}=0.36313$  \\
\hline 
 $\mathbb{Z}_2\ast \mathbb{Z}_{10}$&$x^{6}-2x^{4}-2x^{3}-2x^{2}-2x-1$&$\beta_{10}=1.96595$  & $d_{10}=0.35656$  \\
\hline
 $\mathbb{Z}_2\ast \mathbb{Z}_{11}$&$x^{6}-2x^{4}-2x^{3}-2x^{2}-2x-2$&$\alpha_{11}=1.97781$ & $c_{11}=0.35085$  \\
\hline 
 $\mathbb{Z}_2\ast \mathbb{Z}_{12}$&$x^{7}-2x^{5}-2x^{4}-2x^{3}-2x^{2}-2x-1$ &$\beta_{12}=1.98358$ & $d_{12}=0.34689$  \\
 \hline
  $\mathbb{Z}_2\ast \mathbb{Z}_{13}$&$x^{7}-2x^{5}-2x^{4}-2x^{3}-2x^{2}-2x-2$ &$\alpha_{13}=1.9892$ & $c_{13}=0.34335$  \\
 \hline
  $\vdots$ &$\vdots$ &$\vdots$ & $\vdots$  \\
 \hline
  $\mathbb{Z}_2\ast \mathbb{Z}$&$x^{2}-x-2$& $\alpha_{\infty}=2$ & 0.3333...\\
 \hline
\end{tabular}
\end{center}
\caption{Hecke group, the associated polynomial for the reciprocal geodesics based at the order 2, and the  dominant root and coefficient of the polynomial.} 
\end{table}

 When $k$ is even there are special  reciprocal geodesics based at the order two cone point of 
 $X_{k}=\mathbb{H}/\Gamma_{k}$.  These are non-primitive reciprocal geodesics that start in the order $2$ cone point,  travel to the 
order $k$ cone point, bounce back to the order two cone point and then retrace their steps an even number of times.  These reciprocal geodesics are necessarily non-primitive. We  call their representative words in 
$\Gamma_{k}$, $\mathcal{B}$-type words. They  are simultaneously based at the order 2 cone point and at the order $k$ cone point. 
 For example, the word $ab^{2}ab^{2}$ is a 
$b$-type word in $\mathbb{Z}_{2}*\mathbb{Z}_{4}$. In Lemma  \ref{lem:btypesslow}, we prove that such special reciprocal words  are negligible, allowing us to  conclude the asymptotics in Theorems $\text{A}$   and $\text{B}$ and to precisely determine the asymptotic constants. 
 
\vskip5pt
\noindent {\bf Some background:}  The counting problem for reciprocal geodesics along with  other combinatorial counting problems in the case  of the modular surface, $k=3$,  have been investigated in     \cite{BasSuz}, \cite{BasSuz2}, and 
\cite{BasLiu}.  For the modular group the associated polynomial is $x^{2}-2$ and hence there is no dominant root. The number of reciprocal geodesics of length $2t$ in the modular surface is 
$\frac{1}{4}   \left(\sqrt{2} \right)^{t}\left(1+(-1)^{t}\right)$. 
See  \cite{BasSuz} for the details.                                         
The $k$ odd   and $k=\infty$ cases  for Hecke surfaces  in this paper are  worked out by the second named author in the Doctoral  thesis  \cite{Mar}. For $k$ even, in the paper \cite{DasGon2} the authors determine an asymptotic upper bound for the number of reciprocal geodesics of word length less than $2t$.   Our work was done  independent of theirs. 
The growth rate of the reciprocal geodesics 
  in terms of geometric length  was proven  by   Sarnak \cite{Sar} for the modular group,  and by Erlandsson-Souto 
  \cite{ErSo2}   for  general hyperbolic orbifolds with an even order  cone point. See also \cite{BasSuz, B-K3, DasGon, Er, ErPaSo, ErSo, Mirz, Ri, Tra} for related results and background. For the fundamentals of hyperbolic geometry see \cite{Bus},  and for the basics on combinatorial group theory see \cite{LyndSch, Mag}.  Essential results from analytic combinatorics can be found in  \cite{Fla-Ph}.
  
  \vskip5pt
  \noindent{\bf Plan of the paper:} The proof of Theorem 
  $\text{A}$   is a combination of Theorems
  \ref{thm:binetz2zkoddforrecursion}, \ref{thm:binetz2zkevenforrecursion}, and \ref{thm:growth prim conjugacy special commutators}.   The proof of     
 Theorem  $\text{B}$  is  Theorem  \ref{thm: k even based at order k} in Section \ref{sec: rec. growth based at order k}.
  In section \ref{sec:basics and notation} we set-up notation and state  some of the basics. In section \ref{countingbinarywords}
we outline the general approach to proving Theorems A and B.
In section \ref{sec: normal forms}, we discuss normal forms of the reciprocal words and prove some elementary properties 
they exhibit. Section  \ref{sec: rec. growth based at order 2}
contains  the proofs of  items (1) and (2) of Theorem $\text{A}$, and section  \ref{sec: rec. growth based at order k} has the proof of Theorem $\text{B}$. Section \ref{sec: rec. growth based at order infinity} contains the proof of  item (3) in Theorem A. In section \ref{sec: poly properties} we discuss the properties of the relevant polynomials in Theorems  $\text{A}$ and $\text{B}$, and finally in section \ref{sec: linear system} we   outline  how to find the coefficients in Theorems $\text{A}$ and $\text{B}$,
and finish with an illustrative example of the procedure.


\section{Basics and  notation} \label{sec:basics and notation}

\begin{table} \label{table: notation}
\begin{center}
\begin{tabular}{|c|c| } 
 \hline
 {\bf Definition}& {\bf Notation}\\ 
\hline
Commutator of $x$ and $y$ &$[x,y]$  \\ 
\hline
Reciprocal words  & $\mathcal{R}$   \\ 
\hline
Primitive reciprocal words& $\mathcal{R}^{p}$\\
\hline 
Word length of $g$ & $|g|$\\
\hline
Conjugacy class length of $[g]$ & $|[g]|$ \\
\hline 
Reciprocal words in normal form & $\mathcal{N}$\\
\hline
Set of reciprocal words based at order 2 and order $k$ & $\mathcal{B}$-types \\
\hline
\hline 
\end{tabular}
\end{center}
\caption{Notation Table}
\label{tab:growthrates}
\end{table}
In Table \ref{table: notation} we provide a guide to some of the notation that appears in this paper. 
Consider the group $G=\mathbb{Z}_2\ast \mathbb{Z}_k$, where $3\leq k\leq\infty$.  We use the convention that when $k=\infty$,  $G=\mathbb{Z}_2\ast\mathbb{Z}$. Denote  the generator of $\mathbb{Z}_{2}$   by   $a$ and the generator of  $\mathbb{Z}_k$ by   $b$.  An element $g \in G$ is {\it primitive} if it is not a non-trivial power of another element of $G$. 
The {\it word length}  of $g$,  denoted  $|g|$, is  the minimum length among all words representing $g$ using the symmetric set of generators $\big\{a,b,b^{-1}\big\}$. For the purpose of counting, we use the convention that for each finite $k$ the exponents on $b$ come from the following set of consecutive nonzero integers.

\[E_k=\left\{l \in\mathbb{Z}-\{0\} : -\left\lfloor\frac{k-1}{2}\right\rfloor\leq l \leq \left\lfloor\frac{k}{2}\right\rfloor\right\}\]
 
For example, $E_3=\{-1,1\}$, $E_4=\{-1,1,2\}$, and $E_5=\{-2,-1,1,2\}$.  In addition, when  $k=\infty$  we define $E_{\infty}=\mathbb{Z}-\{0\}$.  Set  $\mathcal{W}= \{$reduced words in the generators of $G\}$, that is,  where the exponent of $a$ is always $+1$ and the exponent of $b$ is always an element of $E_k$. The conjugacy class of $g \in G$ is denoted $[g]$. 
For a positive integer  $s$, since conjugation  commutes with taking powers,  we may define $[g]^{s}:=[g^{s}]$. 
The {\it length of a conjugacy class}  $[g]$ is given by
$|[g]|  = min \{| h | :h\in [g] \}$. A word in 
$\mathcal{W}$ is {\it cyclically reduced} if any cyclic permutation of it is a reduced word. Though cyclically reduced words  in a conjugacy class are not unique they do realize the minimum length in the conjugacy class.  In fact, all conjugates of a cyclically reduced word are cyclic permutations of each other.  For the basics on combinatorial group theory see \cite{LyndSch, Mag}.

We call a reduced word that begins with $a$ and ends with $b^x, x\in E_k$ an $(ab)$-word.  Similarly, we have $(aa)$-, $(bb)$-, $(ba)$-words. We remark the obvious but important fact that an $(ab)$- or $(ba)$-word is cyclically reduced but an $(aa)$-word is not. Some elementary facts  are assembled in the following lemma (see  \cite{BasSuz}    for more details).

 \begin{lem}  \label{lem: conjugates}
 Let $w \in \mathcal{W}$ where $w$ is not conjugate to $a$ or a power of $b$. Then 
 \begin{enumerate}
  \item  $w$ is conjugate to an (ab)-word $y$ with 
  $\lVert w\rVert \geq \lVert y\rVert$.
  \item The only conjugates of the word 
  $ab^{x_0}\dots ab^{x_{n-1}}$, $x_i\in E_k$,  that are   (ab)-words are its even cyclic   permutations. That is, $ab^{x_{n-1}} ab^{x_0}\dots ab^{x_{n-2}}$
  and so on. 
  \item  If $y$ is an (ab)-word and $w^s=y$ for 
  $s$ a positive integer, then $w$ is an (ab)-word  and $s\lVert w\rVert=\lVert y\rVert$.  
  \item   If  $[w]^s=[y]$ then 
  $s \lVert [w]\rVert=\lVert[y]\rVert$.
  \item  In the group $\mathbb{Z}_2\ast \mathbb{Z}_k, 3\leq k\leq\infty$ each infinite order element is a positive  power of a unique, primitive element and every infinite order element is contained in a maximal cyclic subgroup of $G$.

\end{enumerate}
\end{lem}

Using  the  convention that   exponents of $b$ must lie in $E_{k}$, 
we have 
$$
\| w \| = n+ \sum_{i=0}^{n-1}|x_i|,  \text{    for   } w=ab^{x_0}ab^{x_1}\dots ab^{x_{n-1}}.
$$

Note that $n$ is the number of $a's$ that appear in $w$.


\section{Outline of proof}
\label{countingbinarywords}

For $p$ an even order cone point on the Hecke surface $X_{k}$, our interest is in   counting conjugacy classes of reciprocal hyperbolic elements based at $p$ in  the Hecke group 
$\Gamma_{k}$ (see section 1 for the definiton); that is, reciprocal geodesics  based at  $p$  or equivalently 
conjugacy classes of reciprocal words  in the abstract group 
$\mathbb{Z}_2 \ast \mathbb{Z}_k $, for $3\leq k\leq\infty$.

The  counting problem for reciprocal geodesics in a  Hecke surface starts out in a similar fashion as  the counting problem for the modular surface ($k=3$ case).  Namely, we look for a normal form in $\mathbb{Z}_2 \ast \mathbb{Z}_k$ of the conjugacy class that represents a reciprocal geodesic on the surface. This is the set of elements $\mathcal{N}$ discussed in 
Lemma \ref{lem: normal forms}. For $k$ odd a reciprocal word is the power of a reciprocal primitive word, and the number of conjugate elements in normal form is 2.
However, in the  $k \geq 4$ even case
there are two complicating features. The first is that there exist non-primitive reciprocal geodesics  based at $p \in X_{k}$ that are not the power of a primitive reciprocal geodesic based at $p$.  In this case, our standard argument for finding the growth rate of primitive geodesics dominating the non-primitive geodesics breaks down.  One such example  of this phenomenon occurs with  the  word 
$ab^{2}$ in $\mathbb{Z}_{2} \ast  \mathbb{Z}_{4}$.  This word 
is not a reciprocal word but its power  $(ab^{2})^{2}=ab^2ab^{-2}$ 
is a reciprocal word.  Note that $ab^{2}$ corresponds to a closed geodesic that passes through both cone points.

The second complicating feature has to do with 
deciding  how many normal forms are conjugate.
In the case of the  modular surface,  and in fact for $k$ odd,  it is precisely two  (see \cite{BasSuz, Mar}). 
However, for Hecke groups  when  $k$ 
is even there are  normal forms that are not conjugate to any other normal form.  Returning to the   example in 
 $\mathbb{Z}_{2} \ast  \mathbb{Z}_{4}$, the normal form of a reciprocal word based at the order two cone point is 
 $[a,h]$, $h$ a $(bb)$-word. Hence, 
$(ab^{2})^{2}=ab^2ab^{-2}$ is a reciprocal word in normal form and it is easy to see there is no other conjugate one in this form.   On the other hand,  $abab^{-1}$ and $ab^{-1}ab$ are conjugate  reciprocal words  based at the order 2  in normal form. While this may not be  a concern for rough bounds it is a serious concern to determine whether an asymptotic growth rate exists and what it in fact is.

\section{Normal forms}\label{sec: normal forms}

If $k$ is odd then there is only one type of reciprocal geodesic 
on $X_{k}$, namely the ones based at the order 2 cone point. 


\begin{lem}[Normal forms]\label{lem: normal forms}
  Let $a$ be  the generator for  
$\mathbb{Z}_{2}$ and  $b$ the  generator for  $\mathbb{Z}_{k}$. We consider the 
symmetric generating set $\{a, b, b^{-1}\}$ for 
$\mathbb{Z}_{2} \ast \mathbb{Z}_{k}$. Let $p$ be an even order cone point of $X_{k}$.
A  reciprocal geodesic based at $p \in X_{k}$ corresponds to the conjugacy class of a reciprocal hyperbolic  element $w \in \mathbb{Z}_{2} \ast \mathbb{Z}_{k}$ where

\begin{enumerate}
\item if $p$ is the cone point of order two then $w$ is the  product of two distinct  involutions in the conjugacy class of  $a$ 
 and can be  conjugated to the  normal form  $[a,h]$ where
 $h$ is a   (bb)-word of the form
\begin{equation*} 
b^{x_0}\dots ab^{x_{n-1}}
\end{equation*}

\item if $p$ is the cone point of order $k=2m$ then $w$ is the product of two distinct involutions in the conjugacy class of  $b^{m}$ 
and  can be conjugated to the form
$[b^{m}, h]$ where $h$  is an (aa)-word of the form
\begin{equation*}
ab^{x_0}\dots ab^{x_{n-1}}a
\end{equation*}
\end{enumerate}

the exponents   $x_i \in E_k$. 

\end{lem}
The  full set of normal forms for reciprocal geodesics based at the order two cone point is

\begin{equation*} \{[a,h] : h\,\, \textrm{a $(bb)$-word}\}.
\end{equation*} 

and in the $k=2m$ even case, the   full set of normal forms for reciprocal geodesics based at the order $k$  cone point is

\begin{equation*} \{[b^{m},h] : h\,\, \textrm{an $(aa)$-word}\}.
\end{equation*}

Let $p$ be an even order cone point in $X_{k}=\mathbb{H}/\Gamma_{k}.$ Then the set of reciprocal words in 
$G=\mathbb{Z}_2 \ast \mathbb{Z}_k$ representing a reciprocal geodesic based at $p$ is denoted $\mathcal{R}$, and the normal forms are denoted  $\mathcal{N}$.

We first note that if $[a,h]$  is primitive in 
$\mathbb{Z}_2\ast\mathbb{Z}_k$, $h$ a $(bb)$ word, then $h$ is not of order two. 
Similarly, if $[b^{m}, h]$ is primitive in 
$\mathbb{Z}_2\ast\mathbb{Z}_k$, $h$ an $(aa)$ word,  then $h$ is not of order 2.
 
 If  $k=2m$  then a  closed geodesic that passes through both cone points is not necessarily a reciprocal geodesic. For example, 
 $AB^{2}$  in the Hecke group $\Gamma_{4}$ is a hyperbolic element whose axis   passes through both cone points but it does not represent a reciprocal geodesic. 
  In fact, we have

 \begin{lem}\label{lem: passing through both cone points}
 Let $k=2m$ and  $\gamma$  be a closed geodesic on $X_{2m}$ that passes through both cone points. 
 Then $\gamma$ is a reciprocal geodesic based at the cone point of order 2  if and only if it is a reciprocal geodesic based at the order $k$ cone point. 
 Moreover, $\gamma$  corresponds to a word that can be conjugated to  the form $(agb^{m}g^{-1})^{2q}$, where  $g$ is a (ba)-word  or the identity, and $q$ is a positive integer. Equivalently it corresponds to the conjugacy class of a word of the  form $(b^{m}g^{-1}ag)^{2q}$.
 \end{lem}
 
 Words that represent reciprocal geodesics that pass through both cone points we call {\it $\mathcal{B}$-type words.} and we denote them by
 
 \[\mathcal{B}_{2t}:=\{\gamma \subset X_{2m} \text{ a reciprocal geodesic  passing through both cone points}: |\gamma|=2t\}\] 
\[\overset{1-1}{\longleftrightarrow}\{\omega=(agb^mg^{-1})^{2q}:\textrm{$g$ is a $(ba)$-word or id}, q\in\mathbb{N}, |\omega|=2t\},\]   

 \begin{proof}[Proof of Lemma \ref{lem: passing through both cone points}]
 If $w$ represents the hyperbolic element corresponding to 
 $\gamma$ then its axis passes through infinitely many  order $2$ fixed points as well as  infinitely many order $k$ fixed points of elements  in $\Gamma_{k}$.  Moreover, the order two  and $k$ fixed points must alternate along the axis. Hence  a primitive element that generates the stabilizer of this axis is of the form $agb^{m}g^{-1}$. Necessarily it is an even power for otherwise, it is not the product of two conjugate involutions.
 \end{proof}

We next give a rough growth count of  the reciprocal geodesics
in $X_{2m}$ that pass through both cone points.  As noted  in Lemma \ref{lem: passing through both cone points} the length of any such reciprocal geodesic is even. With this in mind we have

\begin{lem} \label{lem:Reciprocal thru both cone points}
For $\alpha > \sqrt{2}$
\begin{equation}
\lim_{t\to\infty}\dfrac{|\mathcal{B}_{2t}|}{\alpha^{t}}=0.
\end{equation}
\end{lem}

\begin{proof}
Set $s=\frac{t-q(m+1)}{2q}$ and 

 \[
 \mathcal{Z}_{2s}=\left\{g:  (ba)-\textrm{word or id } \text{ with }
  |g|=s \right\}
 \]
and note that this set is in 1-1 correspondence with 
$\mathcal{B}_{2t}$ by Lemma \ref{lem: passing through both cone points}.
Now consider a non-identity element $g\in\mathcal{Z}_{2s}$, where $g=b^{x_0}ab^{x_1}a\dots b^{x_{n-1}}a$ and $|g|=s$.  It follows that $\sum_{i=0}^{n-1} |x_i|=s-n$.  Denoting   the number of compositions of the positive integer $r$ with $n$ parts by  $\mathcal{C}_{n}(r)$,  note that there are 
 $\mathcal{C}_{n}(s-n) 2^{n}$ such $g$ since each part can be positive or negative. Using the fact  that 
  $\mathcal{C}_{n}(s-n) \leq 2^{s-n-1}$ and 
  $1 \leq n \leq \lfloor\frac{s}{2}\rfloor$ we have

\begin{equation*}
|\mathcal{B}_{2t}|=|\mathcal{Z}_{2s}|\leq\displaystyle
\sum_{n=1}^{\lfloor\frac{s}{2}\rfloor}\mathcal{C}_{n}(s-n)2^n=\displaystyle\sum_{n=1}^{\lfloor\frac{s}{2}\rfloor}2^{s-n-1}2^n=\displaystyle\sum_{n=1}^{\lfloor\frac{s}{2}\rfloor}2^{s-1}\leq  \left(\frac{s}{4}\right)2^{s}
\leq \left( \frac{t}{8}\right) (\sqrt{2})^{t}
\end{equation*}
where the last inequality follows from 
$s=|g|\leq \frac{t}{2}$.  Finally we have

\[\displaystyle\lim_{t\to\infty}\dfrac{|\mathcal{B}_{2t}|}{\alpha^{t}}=\displaystyle\lim_{t\to\infty}\dfrac{|\mathcal{Z}_{2t}|}{\alpha^{t}}\leq\displaystyle\lim_{t\to\infty}\dfrac{\frac{t}{8}(\sqrt{2})^{t}}{\alpha^t}=\displaystyle\lim_{t\to\infty}\frac{t}{8}\left(\frac{\sqrt{2}}{\alpha}\right)^t=0.\]
since  by assumption $\alpha>\sqrt{2}$.
\end{proof}

\begin{prop}[Based at order 2]  \label{prop: normal form based at order 2}
Consider $G=\mathbb{Z}_2\ast\mathbb{Z}_k$ and suppose $w \in G$ corresponds to  a reciprocal geodesic based at the order 2 cone point.  Then either 
\begin{enumerate}
\item   $w$  is conjugate to a power of a primitive element of the form $[a,g]$, $g$ a (bb)-word,  and it 
     has exactly two  representatives  in the normal form $\mathcal{N}$. Namely, the two conjugates in $\mathcal{N}$ are $[a,h]^n$ and $[a,h^{-1}]^n$,  where $[a, h]$ is primitive in $G$, $n$ is a unique positive integer, and $h$ is a unique (bb)-word not of order 2, or 
      \item $k=2m$ and $w$  is conjugate to the even power of a word of the form $agb^{m}g^{-1}$,  $g$ a (ba)-word or the identity.  That is, $w$ is a $\mathcal{B}$-type reciprocal word. 
 \end{enumerate}
 \end{prop}

  \begin{proof}  If $k$ is odd then a  reciprocal geodesic based at the order two cone point  is the power of a reciprocal geodesic,  and hence has the normal form $[a,g]$ as stated. Similarly, 
  if $k=2m$ is even and the reciprocal geodesic based at the order two does not pass through the order $k$ cone point then it can be put in the normal form $[a,g]$ where $g$ is not conjugate to $b^{m}$.  The fact that there are two conjugates in  this normal form follows in the same way as in the papers \cite{BasSuz, ErSo2, Sar}
  
  On the other hand,  if the reciprocal geodesic based at the order two also passes through the order $k$ cone point ($k=2m$ is necessarily even)  then  by Lemma \ref{lem: passing through both cone points}, it must happen periodically which means $w$ is necessarily the even power of $a gb^{m}g^{-1}$.    That is, an element of $\mathcal{B}$. 
 \end{proof}

     \begin{prop}[Based at order $k$] \label{prop: normal form based at order k}
Consider $G=\mathbb{Z}_2\ast\mathbb{Z}_k$ and suppose $w$ corresponds to  a reciprocal geodesic based at the order $k$  cone point.  Then either 
\begin{enumerate}
 \item $w$  is a power of a primitive element of the form $[b^{m}, g]$, $g$ an (aa)-word, and it
     has exactly two  representatives  in normal form. Namely, the two conjugates in normal form  are $[b^{m},h]^n$ and $[b^{m},h^{-1}]^n$,  where $[b^{m}, h]$ is primitive in $G$, $n$ is a unique positive integer, and $h$ is a unique (aa)-word not of order 2, or
     \item   $w$  is the even power of a word of the form $b^{m}gag^{-1}$, $g$ is an (ab)-word or the identity. That is, $w$ is a $\mathcal{B}$-type reciprocal word. 
   \end{enumerate}
  \end{prop} 
  
  We omit the proof of  Proposition  \ref{prop: normal form based at order k}   as the statement is simply a recasting of  
  Proposition  \ref{prop: normal form based at order 2}
 with an identical proof. 
 
 \begin{rem}\label{rem: simultaneously in both}
 When $k=2m$ is even it is a consequence of  Propositions  \ref{prop: normal form based at order 2} and 
 \ref{prop: normal form based at order k}
that there exist 
 reciprocal geodesics that are simultaneously based at the order 2 cone point and the order $k$ cone point. For example,  the element
 $(ab^{m})^2 =ab^{m}ab^{m}$ is conjugate to 
 $(b^{m}a)^{2}=b^{m}ab^{m}a$. These are the $\mathcal{B}$-type elements. 
 \end{rem}

  \section{Reciprocal Growth based at the order 2 cone point}
   \label{sec: rec. growth based at order 2}
  In this section, we consider reciprocal growth based at the order two cone point. There are two cases to consider depending one whether $k$ is odd or even. 

  \subsection{Reciprocal growth based at the order two cone point, $k$ odd}
  The results in this section appear in the Doctoral  dissertation   \cite{Mar}.
  When $k$ is odd there is only one type of reciprocal geodesic on the Hecke surface $X_{k}$. Namely, the reciprocal geodesics  based at the order 2 cone point. 
  Let $\mathcal{N}_{2t}$ be the set of normal forms for reciprocal words of length $2t$. 
  
  \begin{prop}\label{thm:recurrencerelationkodd}
  	For $\mathbb{Z}_{2} * \mathbb{Z}_{k}$ and $k=2m+1$, the number of reciprocal words based at the order 2  in normal form of length $2t$, $|\mathcal{N}_{2t}|$, satisfies the following:
  	
  	\begin{itemize}
  		\item [(i)]  for $2 \leq t \leq m + 1$, $|\mathcal{N}_{2t}| = \frac{1}{3}(2^{t}+2(-1)^{t})$
  		\item [(ii)]  for $t = m + 2$, $|\mathcal{N}_{2t}| = \frac{1}{3}(2^{t}+2(-1)^{t}) - 2$
  		\item [(iii)]  for $t > m + 2$, 
\begin{equation}\label{eq: recurrence relation}
|\mathcal{N}_{2t}| = 2|\mathcal{N}_{2(t-2)}| + 2|\mathcal{N}_{2(t-3)}| + \dots + 2|\mathcal{N}_{2(t-(m+1))}|
\end{equation}
  		
  	\end{itemize}

  \end{prop}

Recall that our allowable exponents for $k=2m+1$ are,

\begin{displaymath}
E_k=\{l \in \mathbb{Z}-\{0\} : -m \leq l \leq m \}
\end{displaymath}

  \begin{proof} Proofs of $(i)$ and $(ii)$  we leave to the reader. 
  	
  	Proof of $(iii)$.  Assume $t > m+2$. To derive this  recurrence relation we first  consider normal forms as in 
	item (1) of Lemma \ref{lem: normal forms}  according to whether they start with 
$ab , ab^{2} , ab^{3}, ..., ab^{m}$. Hence the number of possible normal forms of length $2t$  that start with $ab$ are in one to one correspondence with the normal forms of length $2(t-2)$.
Similarly, the ones that start with $ab^{2}$ are in one to one correspondence with the normal forms of length $2(t-3)$, and so on
until the ones that start with $ab^{m}$ correspond to the normal forms of 
length $2(t-(m+1))$.  Finally, the normal forms that start with 
$ab^{-1} , ab^{-2} , ab^{-3}, ..., ab^{-m}$ have the same  count as the ones that start with the analogous positive integer. Hence,  we have derived the recurrence  relation
$$
|\mathcal{N}_{2t}| = 2|\mathcal{N}_{2(t-2)}| + 2|\mathcal{N}_{2(t-3)}| + \dots + 2|\mathcal{N}_{2(t-(m+1))}|
$$  	
  \end{proof}

  \begin{thm}\label{thm:binetz2zkoddforrecursion}
   The conjugacy classes of reciprocal words based at the order 2  in 
   $\mathbb{Z}_{2} * \mathbb{Z}_{k}$,  for  $k=2m+1$ has growth rate

\begin{displaymath}
\frac{\big|\{\gamma \text{ reciprocal geodesic based at order 2}: |\gamma|=2t\}\big|}
{\alpha_{2m+1}^{t}} \xrightarrow{t\to\infty} \frac{c_{2m+1}}{2}
\end{displaymath}

   where  $\alpha_{2m+1}$  is the unique positive real root of the polynomial, $x^{m+1} -2x^{m-1} -2x^{m-2} - \dots -2x^{2} -2x-2$, and $c_{2m+1}$ 	is the coefficient of $\alpha_{2m+1}$ in the solution of the linear system determined by the recurrence relation (\ref{eq: recurrence relation}).
   The growth rate for the primitive conjugacy classes  of reciprocal words 
   is the same. 
  \end{thm}

  \begin{proof} Recall $m$ is a positive integer. 
  We begin with the recurrence relation
  	
  	\begin{equation*}
  	|\mathcal{N}_{2t}| = 2|\mathcal{N}_{2(t-2)}| + 2|\mathcal{N}_{2(t-3)}| + \dots + 2|\mathcal{N}_{2(t-(m+1))}|
  	\end{equation*}
  	
  	and rewrite it as 
  	
  	\begin{equation*}
  	|\mathcal{N}_{2t}| - 2|\mathcal{N}_{2(t-2)}| - 2|\mathcal{N}_{2(t-3)}| - \dots - 2|\mathcal{N}_{2(t-(m+1))}| = 0
  	\end{equation*}
  	Hence,  the characteristic polynomial is
	\begin{equation}
  	P_{2m+1}(x) = x^{m+1}-2\sum_{j=1}^{m} x^{m-j}     
	=x^{m+1} -2x^{m-1} - \dots -2x-2 
  	\end{equation}
	
The polynomial $P_{2m+1}$, by Lemma 
\ref{lem: dominant root}, has a unique positive root which is dominant. This root  is denoted  $\alpha_{k}$, ($k=2m+1$).
 	Using the fact that there are exactly two conjugate words in normal form (item (1) of  Proposition \ref{prop: normal form based at order 2} and the fact that there are no 
	$\mathcal{B}$-type words), we obtain the  result

\begin{equation}
\frac{\big|\{\gamma \text{ reciprocal geodesic based at order 2}: |\gamma|=2t\}\big|}
{\alpha_{2m+1}^{t}} \xrightarrow{t\to\infty} \frac{c_{2m+1}}{2}
\end{equation}

To finish the proof, we are left to show  the analogous result for the primitive reciprocal   words  in $\mathbb{Z}_{2} * \mathbb{Z}_{k}$,  for $k$ odd. Denote  the nonprimitive reciprocal words based at the order 2 cone point by $|R_{2t}^{np}|$.

 \vskip5pt

  \noindent{\bf  Claim:} 
  For  $\mathbb{Z}_{2} * \mathbb{Z}_{k}$ and $k$ odd, we have
  	$|R_{2t}^{np}| \leq \gamma_{1} t \alpha_{2m+1}^{t/2}$ 
	where   $\gamma_{1}$ is  a universal constant.
  \vskip5pt

\noindent{\bf Proof of Claim:} Denoting the roots of 
$P_{2m+1}$  by 
$\lambda_1$,..., $\lambda_{m+1}$ (they are distinct by Lemma \ref{lem: dominant root}), 
we first note that   $|\mathcal{N}_{2s}|$ can be expressed as the  linear combination 
\begin{equation}\label{eq: general solution}
|\mathcal{N}_{2s}|
=c_{1}\lambda_{1}^{s} +...+c_{m+1}\lambda_{m+1}^{s}
\end{equation}

Now since $\alpha_{2m+1}$ is the dominant root, we obtain the  the upper bound
 \begin{equation}\label{eq: upper bound}
 \frac{|\mathcal{N}_{2s}| }{2}\leq c_{k} \alpha_{2m+1}^{s}
 \end{equation}.
 
Fix $t \in \mathbb{Z}_{+}$. Note that  the  nonprimitive reciprocal words  of length   $2t$ are in one to one correspondence  with the primitive reciprocal word of length $2s$, where $s$ is a proper divisor of  $t$. 
See \cite{BasSuz}  for  details. Summing over proper divisors $s$ of $t$ we have
  	
  	\begin{equation*} \label{nonprims computation1 Z2Zk}
  	|R_{2t}^{np}| = 
\sum |\text{$ \gamma$ a prim. rec. geodesic based at order 2: $|\gamma|=2s$}| 
\end{equation*}
\begin{equation*}
\leq \sum |\text{$ \gamma$ a  rec. geodesic based at order 2: $|\gamma|=2s$}|  
\leq c  \alpha_{2m+1}^{t/2} \sum 1 \leq \frac{ct \alpha_{2m+1}^{t/2}}{2} 
  \end{equation*}
where the second from right inequality follows from  inequality 
(\ref{eq: upper bound}) and the fact that the largest divisor is $t/2$, and the last inequality follows from the fact that $t$ has at most 
$\frac{t}{2}$ divisors. Setting  $\gamma_{1}=\frac{c}{2}$  yields  the desired inequality and finishes the proof of Theorem \ref{thm:binetz2zkoddforrecursion}.
\end{proof}


 \subsection{Reciprocal growth based at the order 2 cone point, $k$ even}
 Our goal in this subsection is to compute the growth rate of the  conjugacy classes of reciprocal words  based at the order two cone point in $\mathbb{Z}_{2} * \mathbb{Z}_{k}$ for 
 $k=2m$ even. To this end,
 we first compute the growth  rate of all normal forms including $\mathcal{B}$-type words, and then  we  show   that the normal forms of $\mathcal{B}$-type words are negligible  as the  word  length goes to infinity.

\begin{prop}\label{thm:recurrence relation even k}
	For $G=\mathbb{Z}_{2} * \mathbb{Z}_{k}$ and  with 
	$k=2m$ even, the number of reciprocal words based at the order 2  in normal form of length $2t$, 
	$|\mathcal{N}_{2t}|$, satisfies the following:
	
	\begin{itemize}
		\item [(i)]  for $2 \leq t < m + 1$, $|\mathcal{N}_{2t}| = \frac{1}{3}(2^{t}+2(-1)^{t})$
		\item [(ii)]  for $t = m + 1$, $|\mathcal{N}_{2t}| = \frac{1}{3}(2^{t}+2(-1)^{t}) - 1$
		\item [(iii)]  for $t = m + 2$, $|\mathcal{N}_{2t}| = \frac{1}{3}(2^{t}+2(-1)^{t}) - 2$
		\item [(iv)]  for $t > m + 2$, 

\begin{equation}\label{eq: normal form recursion}
\mathcal{N}_{2t} = 2|\mathcal{N}_{2(t-2)}|  + 2|\mathcal{N}_{2(t-3)}|  \dots + 2|\mathcal{N}_{2(t-m)}| + |\mathcal{N}_{2(t-(m+1))}|
\end{equation}
		
	\end{itemize}

\end{prop}

Recall that for counting purposes the  allowable exponents for $k=2m$ are,

\begin{displaymath}
E_k=\{l \in \mathbb{Z}-\{0\} : -m+1 \leq l  \leq m \}
\end{displaymath}

\begin{proof} Recall the normal form in Lemma \ref{lem: normal forms}, item (1).The proofs of $(i)$,$(ii)$ and $(iii)$ are straightforward and left  to  the reader.  The proof of item $(iv)$ is essentially the same as for  the odd case. The only difference is that since
$ab^{m}=ab^{-m}$,  hence the last term in the recurrence relation is not doubled since $b^{m}=b^{-m}$.  \end{proof}

Our goal  is to prove

\begin{thm}\label{thm:binetz2zkevenforrecursion}
The conjugacy classes of reciprocal words in 
   $G=\mathbb{Z}_{2} * \mathbb{Z}_{k}$,  for  $k=2m$ has growth rate
  \begin{equation}
\frac{|\{\gamma \text{ reciprocal geodesic based at order } 2:
|\gamma|=2t\}|}
{\beta_{2m}^{t}} \xrightarrow{t\to\infty} \frac{d_{2m}}{2}
\end{equation}
where  $\beta_{2m}$ is  the unique positive real root of the polynomial, $x^{m+1} -2x^{m-1} -2x^{m-2} - \dots -2x^{2} -2x-1$ and $d_{2m}$ is the coefficient of $\beta_{2m}$ in the 
   linear recurrence relation (\ref{eq: normal form recursion}).
   The growth rate for the primitive conjugacy class of reciprocal words 
   is the same. 
  \end{thm}

We start with the following

\begin{lem}\label{lem: recurrence asymptotic}
$|\mathcal{N}_{2t}| \sim d_{2m} \beta_{2m} ^{t}$, as $t \to \infty$. Where $d_{2m}$ is the corresponding   coefficient for $\beta_{2m}$.
\end{lem}

\begin{proof} Assume $t >m+2$, and hence by item (iv) of 
  Proposition \ref{thm:recurrence relation even k}
the relevant recurrence relation is 
\begin{equation*}
	|\mathcal{N}_{2t}| - 2|\mathcal{N}_{2(t-2)}| - 2|\mathcal{N}_{2(t-3)}| - \dots - 2|\mathcal{N}_{2(t-m)}| - |\mathcal{N}_{2(t-(m+1))}| = 0
	\end{equation*}
	and associated characteristic polynomial is
	\begin{equation}
	x^{m+1} -2x^{m-1} - \dots -2x-1.
	\end{equation}
By Lemma \ref{lem: dominant root} there is a unique
	   positive real root which is dominant  for this polynomial which we denote by   $\beta_{2m}$.  Thus as $t \to \infty$, 
	   $|\mathcal{N}_{2t}| \sim d_{2m} \beta_{2m} ^{t}$ where $d_{2m}$ is the    corresponding coefficient of $\beta_{2m}$ in the general solution  of the recurrence relation (\ref{eq: normal form recursion}).
  \end{proof}	

Recall that in the $k=2m$ even case,  there are two types of  normal form elements for the reciprocal words based at the order 2 cone point.  Denote the normal forms of length $2t$ by $\mathcal{N}_{2t}$.
In order to produce the number of conjugacy classes  of reciprocal geodesics based at the order 2 (for $k$ even) we need to divide  $|\mathcal{N}_{2t}|$  by  the number in normalized form that are conjugate.   As we saw in 
Proposition  \ref{prop: normal form based at order 2} there are two types of words in $\mathcal{N}$.  The ones that are the power of 
an element of the form $[a,h]$, $h$ a (bb) word, and the words in    
$\mathcal{B}_{2t}$  which are    reciprocal words that are an even power of a word of the form $agb^{m}g^{-1}$,
where $g$ is a $(ba)$ word or the identity.

\begin{lem}[$\mathcal{B}$-type words grow slowly]
\label{lem:btypesslow}
If $k$ is even and $p \in X_{k}$ is a cone point, then the set of 
$\mathcal{B}$-type words are negligible (as $t \rightarrow \infty$) in the set $\mathcal{N}$. That is,

\[\displaystyle\lim_{t\to\infty}\dfrac{|\mathcal{B}_{2t}|}{|\mathcal{N}_{2t}|}=0.\]

\begin{equation}
\label{eqn:atypesdominate}
|\mathcal{N}_{2t}|\sim |\mathcal{N}_{2t}-\mathcal{B}_{2t}|, \textrm{as}\,\,t\to\infty.\end{equation}

\end{lem}

\begin{proof}
We know  $|\mathcal{N}_{2t}|\sim d_{2m}\beta_{2m}^t$, as $t\to\infty$, where $\beta_{2m} > \sqrt{2}$ and hence it follows from Lemma \ref{lem:Reciprocal thru both cone points}
that 

\[\displaystyle\lim_{t\to\infty}\dfrac{|\mathcal{B}_{2t}|}{|\mathcal{N}_{2t}|}=0,\]

and thus
\begin{equation}
|\mathcal{N}_{2t}|\sim |\mathcal{N}_{2t}-\mathcal{B}_{2t}|, \textrm{as}\,\,t\to\infty.\end{equation}
\end{proof}

\begin{proof}[Proof of Theorem \ref{thm:binetz2zkevenforrecursion}]
For notational convenience
set 
$$
R_{2t}
=\{\gamma \text{ a reciprocal geodesic based at order } 2: |\gamma|=2t \}.
$$

From Proposition \ref{prop: normal form based at order 2},  item (1), we know that there are two   normal forms that are conjugate  in 
$\mathcal{N}_{2t}-\mathcal{B}_{2t}$ and hence represent the same element in $R_{2t}$.
Thus
\begin{equation}\label{eqn:conjclass}
\frac{1}{2}|\mathcal{N}_{2t}-\mathcal{B}_{2t}| \leq |R_{2t}| \leq \frac{1}{2}|\mathcal{N}_{2t}-\mathcal{B}_{2t}|+|B_{2t}|.\end{equation}

Dividing  by $|\mathcal{N}_{2t}|$:

\[\frac{1}{2}\frac{|\mathcal{N}_{2t}-\mathcal{B}_{2t}|}{|\mathcal{N}_{2t}|}\leq
\frac{|R_{2t}|}{|\mathcal{N}_{2t}|}\leq \frac{1}{2}\frac{|\mathcal{N}_{2t}-\mathcal{B}_{2t}|}{|\mathcal{N}_{2t}|}+\frac{|\mathcal{B}_{2t}|}{|\mathcal{N}_{2t}|}.\]

Letting  $t\to\infty$ and using Lemma  \ref{lem:btypesslow} and the asymptotic (\ref{eqn:atypesdominate}), we have \[\displaystyle\lim_{t\to\infty}\frac{|R_{2t}|}{|\mathcal{N}_{2t}|}=\frac{1}{2}.\]
Using  Lemma
\ref{lem: recurrence asymptotic} we obtain

 \begin{equation}
\frac{|R_{2t}|}{\beta_{2m}^{t}}=\frac{|R_{2t}|}{|\mathcal{N}_{2t}|}
\frac{|\mathcal{N}_{2t}|}{\beta_{2m}^{t}}
\xrightarrow{t\to\infty} \frac{d_{2m}}{2}.
\end{equation}

Finally the same asymptotic  holds for the primitive reciprocal geodesics based at the order 2 cone point since the 
words in $\mathcal{B}$-type ones are negligible and the non-primitve  elements in $\mathcal{N}_{2t}-\mathcal{B}_{2t}$ are powers of primitive elements in $\mathcal{N}_{2t}-\mathcal{B}_{2t}$ as in the $k$ odd case.
\end{proof}

\section{Reciprocal Growth  based at the order $k$ 
cone point  in $X_k$ for $k$ even.}
\label{sec: rec. growth based at order k}

We next take up the case of reciprocal geodesics based at the order $k$ cone point of $X_{k}$. Note that  $k$ is necessarily even; Set $k=2m$.  
   Such a reciprocal word is the product of two involutions each conjugate to $b^{m}$. Recall from Lemma \ref{lem: normal forms}, item (2), 
after conjugation we can put this word in 
   the reduced normal form $[b^{m},h]$, where $h$ is an (aa) word. In this section, the set of all such reduced normal forms is denoted $\mathcal{N}$. Now if $2t=|[b^{m},h]|$, then $t=m+|h|$ and 
   
\begin{equation}\label{eq: normal form for reciprocal based at order k}
h=ab^{x_0}...ab^{x_{n-1}}a
\end{equation}
where $x_0,...x_{n-1}$ are in the subset of integers 
\begin{displaymath}
E_k=\{l  \in \mathbb{Z}-\{0\} : -m+1 \leq l \leq m \}.
\end{displaymath}
 Thus the  normal form $[b^{m},h]$  is determined by the $(aa)$-word, $h$, and note that $|h|=0,3,4,...$.
Let $\mathcal{N}_{2t}$ be the number of reciprocal words in normal form, and note that $t \geq m+1$. Counting the elements of $\mathcal{N}_{2t}$ is equivalent to counting the number of $(aa)$ words $h=ab^{x_0}...ab^{x_{n-1}}a$ where 
$|h|=t-m$.  That is,
$$
\mathcal{N}_{2t} \overset{1-1}{\longleftrightarrow}
 \{h: h \text{ an (aa) word}, |h|=t-m\}
$$

\begin{prop}\label{thm:recurrence relation even k based at order k}
	For $G=\mathbb{Z}_{2} * \mathbb{Z}_{k}$ and $k=2m$ with $k$ even, the number of reciprocal words  based at the order $k$ in normal form of length $2t$, 
	$|\mathcal{N}_{2t}|$, satisfies the following:
	\begin{itemize}
		\item [(i)]  for $t = m + 1$, $|\mathcal{N}_{2t}| =1$
		\item [(ii)]  for $t = m + 2$, $|\mathcal{N}_{2t}| = 0$
		\item [(iii)]  for $t = m + 3$, $|\mathcal{N}_{2t}| = 2$
		\item [(iv)]  for $t > m + 3$, 
	\begin{equation}\label{eq: recurrence for rec. word based at k}
	|\mathcal{N}_{2t}| = 2|\mathcal{N}_{2(t-2)}|  + 2|\mathcal{N}_{2(t-3)}|  \dots + 
2|\mathcal{N}_{2(t-m)}| + |\mathcal{N}_{2(t-(m+1))}|
\end{equation}
		
	\end{itemize}
\end{prop}

\begin{proof} The proofs of $(i)-(iii)$ are left to the reader to check. 
For item $(iv)$,  $t>m+3$, the proof follows  in the same way as in  Propositions \ref{thm:recurrencerelationkodd} and \ref{thm:recurrence relation even k} with a few modifications. 
To   derive this  recurrence relation  first  consider  $h$  in expression
\ref{eq: normal form for reciprocal based at order k}.  Now, we count the number of
$h$ that start with $ab^{1}$. There are $\mathcal{N}_{2(t-2)}$ 
and since there are the same number for  $ab^{-1}$, there are
there are 
$2|\mathcal{N}_{2(t-2)}|$. Similarly, we count the number of 
$h$ that start with   
$ab^{2},  ab^{-2}, ab^{3},ab^{-3}, ..., ab^{m-1}, ab^{-(m-1)}$.
Finally,  since $ab^{m}=ab^{-m}$,  the last term in the recurrence relation is not doubled. Hence,  we have derived the recurrence  relation
$$
|\mathcal{N}_{2t}| = 2|\mathcal{N}_{2(t-2)}|  + 2|\mathcal{N}_{2(t-3)}|  \dots + 
2|\mathcal{N}_{2(t-m)}| + |\mathcal{N}_{2(t-(m+1))}|
$$ 
\end{proof}

\begin{lem}\label{lem: recurrence asymtotic for base at k}
$|\mathcal{N}_{2t}| \sim e_{2m} \beta_{2m} ^{t}$, as $t \to \infty$. 
Where  $e_{2m}$ is the coefficient of $\beta_{2m}^{t}$ in the solution of the linear system determined by  the recurrence relation (\ref{eq: recurrence for rec. word based at k}).
\end{lem}

\begin{proof}
The characteristic polynomial corresponding to the linear recurrence relation 
(\ref{eq: recurrence for rec. word based at k}) is
$$
x^{m+1}-2x^{m-1}-2x^{m-2}-...-2x-1
$$
Hence by Lemma   \ref{lem: dominant root}, the dominant root is $\beta_{2m}$ and $e_{2m}$ is  its coefficient.
\end{proof}

We next show that the $\mathcal{B}$-type words are negligible
in the set of normalized reciprocal words based at the order $k$ cone point. 

\begin{lem}[$\mathcal{B}$-type words grow slowly]
\label{lem:Btypebased at k small}
\[\displaystyle\lim_{t\to\infty}\dfrac{|\mathcal{B}_{2t}|}{|\mathcal{N}_{2t}|}=0.\]

\begin{equation}
\label{eqn:atypesdominate for base k}
|\mathcal{N}_{2t}|\sim|\mathcal{N}_{2t}-\mathcal{B}_{2t}|, \textrm{as}\,\,t\to\infty.
\end{equation}
\end{lem}

\begin{proof} This proof follows closely the proof of Lemma
\ref{lem:btypesslow}.
By Lemma \ref{lem: recurrence asymtotic for base at k} we know  $|\mathcal{N}_{2t}| \sim e_{2m} \beta_{2m} ^{t}$ and since  $\beta_{2m} >\sqrt{2}$  we can conclude by  Lemma \ref{lem:Reciprocal thru both cone points} that 
\[\displaystyle\lim_{t\to\infty}\dfrac{|\mathcal{B}_{2t}|}{|\mathcal{N}_{2t}|}=0.\]

The second asymptotic follows 
from  the fact that  $\mathcal{N}_{2t}$ is the disjoint union of 
 $\mathcal{N}_{2t}-\mathcal{B}_{2t}$   and $\mathcal{B}_{2t}$.\end{proof}

\begin{thm}\label{thm: k even based at order k}
Suppose   
$X_{k}$ is the $(2,k,\infty)$ Hecke surface with $k=2m$  even. Then
\begin{equation}
\frac{\big|\{\gamma \text{ reciprocal geodesic based at order 
$k$ cone point}: |\gamma|=2t\}\big|}
{\beta_{2m}^{t}} \xrightarrow{t\to\infty} \frac{e_{2m}}{2}
\end{equation}
where $\beta_{2m}$ is the unique positive root of the polynomial

\begin{equation}\label{eq: poly based at k}
  x^{m+1} -2x^{m-1} -2x^{m-2}- \dots -2x-1 , \text{ for k=2m}   
\end{equation}
and $e_{2m}$ is the coefficient of $\beta_{2m}^{t}$ in the solution of the linear system determined by  the recurrence relation (\ref{eq: recurrence for rec. word based at k}). The same asymptotic holds for the primitive reciprocal geodesics based at the order $k$ cone point. 

\end{thm}

\begin{proof}
This proof follows fairly closely the proof  Theorem \ref{thm:binetz2zkevenforrecursion}.
Let $R_{2t}$ be the set of reciprocal geodesics based at the cone point of order $2m$  having  length $2t$.
Equivalently these are conjugacy classes of  reciprocal hyperbolic elements  which can be written as the product of two involutions each conjugate to $B^{m}$ in $\Gamma_{2m}$, $k=2m$. 

According to  Proposition \ref{prop: normal form based at order k} item (1),  there are two   normal forms that are conjugate  in 
$\mathcal{N}_{2t}-\mathcal{B}_{2t}$ and thus  represent the same element in $R_{2t}$.
Hence
\begin{equation}\label{eqn:conjclass}
\frac{1}{2}|\mathcal{N}_{2t}-\mathcal{B}_{2t}| \leq |R_{2t}| \leq \frac{1}{2}|\mathcal{N}_{2t}-\mathcal{B}_{2t}|+|B_{2t}|.\end{equation}

Dividing  by $|\mathcal{N}_{2t}|$:

\[\frac{1}{2}\frac{|\mathcal{N}_{2t}-\mathcal{B}_{2t}|}{|\mathcal{N}_{2t}|}\leq
\frac{|R_{2t}|}{|\mathcal{N}_{2t}|}\leq \frac{1}{2}\frac{|\mathcal{N}_{2t}-\mathcal{B}_{2t}|}{|\mathcal{N}_{2t}|}+\frac{|\mathcal{B}_{2t}|}{|\mathcal{N}_{2t}|}.\]

Finally letting  $t\to\infty$ and using  the two asymptotics in Lemma  
\ref{lem:Btypebased at k small}, we have \[\displaystyle\lim_{t\to\infty}\frac{|R_{2t}|}{|\mathcal{N}_{2t}|}=\frac{1}{2}.\]
Lastly by  Lemma \ref{lem: recurrence asymtotic for base at k}

 \begin{equation}
\frac{|R_{2t}|}{\beta_{2m}^{t}}=\frac{|R_{2t}|}{|\mathcal{N}_{2t}|}
\frac{|\mathcal{N}_{2t}|}{\beta_{2m}^{t}}
\xrightarrow{t\to\infty} \frac{e_{2m}}{2}
\end{equation}

The same asymptotic  holds for the primitive reciprocal geodesics based at the order $k$  cone point since the $\mathcal{B}$-type ones by Lemma \ref{lem:Reciprocal thru both cone points} 
 are negligible since $\beta_{2m} >\sqrt{2}$ and the non-primitive  words in $\mathcal{N}_{2t}-B_{2t}$ are powers of primitive
  $\mathcal{N}_{2t}-B_{2t}$ (as in the $k$ odd case).
\end{proof}


\section{Growth in $\mathbb{Z}_2\ast \mathbb{Z}$}
\label{sec: rec. growth based at order infinity}

In this section we count conjugacy classes of reciprocal words in $G = \mathbb{Z}_2\ast \mathbb{Z}$.  Let $a$ be a generator of $\mathbb{Z}_2$ and we let $b$ be a generator of $\mathbb{Z}$. Since reciprocal words have even length we use the parameter $2t$ for  ease of  computation.

\begin{prop}\label{prop:recurrence}
	Let $G=\mathbb{Z}_2\ast\mathbb{Z}$.  Then  
	$|\mathcal{N}_{2t}|=\frac{1}{3}(2^{t}+2(-1)^{t})$.
\end{prop}

\begin{proof}
	We begin by showing that the following recurrence relation holds for reciprocal words in $G$.

\begin{equation}\label{eqn:recurrence}
|\mathcal{N}_{2t}|=|\mathcal{N}_{2(t-1)}|+2|\mathcal{N}_{2(t-2)}|
\end{equation}

To prove this, we consider a mapping, $\varphi: \mathcal{N}_{2t}\to\mathcal{N}_{2(t-1)}\cup\mathcal{N}_{2(t-2)}$.  

Let $w=ab^{x_1}ab^{x_2}\dots ab^{x_n}ab^{-x_n}\dots ab^{-x_2}ab^{-x_1}\in\mathcal{N}_{2t}$.  Then we define $\varphi$ as the following piecewise function.

\begin{equation}\label{eqn:phimap}
\varphi(w)=
\begin{cases}
ab^{x_1-1}ab^{x_2}\dots ab^{x_n}ab^{-x_n}\dots ab^{-x_2}ab^{-x_1+1}\in\mathcal{N}_{2(t-1)}, & \text{if}\,\, x_1>1\\
ab^{x_1+1}ab^{x_2}\dots ab^{x_n}ab^{-x_n}\dots ab^{-x_2}ab^{-x_1-1}\in\mathcal{N}_{2(t-1)}, & \text{if}\,\, x_1<-1\\
ab^{x_2}\dots ab^{x_n}ab^{-x_n}\dots ab^{-x_2}\in\mathcal{N}_{2(t-2)}, & \text{if}\,\, x_1=\pm 1
\end{cases}
\end{equation}

The mapping $\varphi$ is surjective and the first two branches  of the map are injective, while the last branch  is clearly 2-1.  This establishes (\ref{eqn:recurrence}).

	Following the standard method for recurrence relations, we use equation (\ref{eqn:recurrence}) to get the corresponding characteristic equation,  
	
	\begin{equation*}
	x^2 - x - 2 = 0.
	\end{equation*}
	
	The roots are $ x =-1 $ and $ x =2  $, and this yields the closed form
	
	\begin{equation*}
	|\mathcal{N}_{2t}| = c_12^t + c_2(-1)^t.
	\end{equation*}
	
	We then use initial values of the sequence ($|\mathcal{N}_2|=0$ and $|\mathcal{N}_4|=2$) to solve for $c_1$ and $c_2$ and we get,
	
	\begin{equation}\label{eqn:recwordcount}
	|\mathcal{N}_{2t}|=\frac{1}{3}\left(2^t+2(-1)^t\right).
	\end{equation}
	\end{proof}

Our goal in this section is to prove

\begin{thm}\label{thm:growth prim conjugacy special commutators}
	Let  $R$ be the set of  conjugacy classes  of reciprocal words in $\mathbb{Z}_{2}\ast \mathbb{Z}$.   Then
\begin{equation}
	 |R_{2t}| \sim \frac{1}{6} 2^{t}, \text{  as  } t \rightarrow \infty,
	\end{equation}
and the primitive reciprocal conjugacty classes have the same growth rate. 	
\end{thm}

Denote the primitive and non-primitive reciprocal geodesics of length $2t$ by $R^{p}_{2t}$  and $R^{np}_{2t}$, respectively.

\begin{lem}\label{lem:nonprimitive bound} For  $G=\mathbb{Z}_2\ast \mathbb{Z}$,
	\begin{equation}
	|R^{np}_{2t}|\leq  \frac{1}{6} t  2^{t/2}.
	\end{equation}        
\end{lem}

\begin{proof}
	Fix $t$ a positive integer. All sums in this proof are over the proper divisors of $t$. Using the fact that non-primitive reciprocal words of length 
	$2t$ are in one to one correspondence with primitive reciprocals of word length $2s$ where $s$ divides $t$  we have,
	\begin{equation}\label{eq:nonprimitive computation1}
	|R^{np}_{2t}|=\sum |R^{p}_{2s}| \leq \sum |R_{2s}|=\sum \frac{1}{6}(2^{s}+2(-1)^{s})  \leq \frac{2}{6} \sum 2^{s} 
	\end{equation}
	Using the fact that the largest possible proper divisor of $t$ is $s=t/2$, we have that the right hand side of equation (\ref{eq:nonprimitive computation1}) is,
	\begin{equation}\label{eq:nonprimitive computation2}
	\leq   \frac{2}{6}   2^{t/2}\sum 1 \leq  \frac{1}{6} t  2^{t/2},
	\end{equation}
	where the last inequality follows from the fact that the most number of proper divisors an integer $t$ can have is $t/2$.

\end{proof}

\begin{proof}[Proof of Theorem \ref{thm:growth prim conjugacy special commutators}]

Using item (1) of Proposition \ref{prop: normal form based at order 2},  we know that each normal form has a conjugate. Hence, we divide  equation (\ref{eqn:recwordcount}) by 2  to  obtain the growth rate of the conjugacy classes of reciprocal words.  

To prove the same asymptotic for the primitive reciprocal words, we apply  Lemma  \ref{lem:nonprimitive bound}  to get

\begin{equation}\label{eq:bounds1}
|R_{2t}|-\frac{1}{6} t  2^{t/2} \leq  |R^{p}_{2t}|\leq  |R_{2t}|,
\end{equation}

and hence,

\begin{equation}\label{eq:bounds2}
1-\frac{\frac{1}{6} t  2^{t/2}}{|R_{2t}|} \leq \frac{|R^{p}_{2t}|}{|R_{2t}|} \leq 1.
\end{equation}

Using this inequality and the fact  that $\frac{\frac{1}{6} t  2^{t/2}}{|R_{2t}|} \longrightarrow 0$ as $t \longrightarrow \infty$ finishes the proof of 
Theorem \ref{thm:growth prim conjugacy special commutators}.
\end{proof}


\section{Properties of the polynomial families} \label{sec: Summary}\label{sec: poly properties}

Consider the two polynomial families,

\begin{equation}\label{eq: polynomial and roots}
\left\{
\begin{array}{lr }
    P_{2m+1}(x)=   x^{m+1} -2x^{m-1}-2x^{m-2} - \dots -2x-2, \text{ for $m\geq 2, k=2m+1$}    \\
  Q_{2m}(x)=x^{m+1} -2x^{m-1} -2x^{m-2}- \dots -2x-1, \text{ for $m \geq 2, k=2m$}  \\
\end{array}
\right\}
\end{equation}

\begin{lem} \label{lem: dominant root}
Each of the polynomials $Q_{2m}$ and  $P_{2m+1}$  in expression (\ref{eq: polynomial and roots})  have distinct  complex roots  and   a unique positive real root which is dominant over its other  roots.  The polynomial $P_{2m+1}$ is irreducible over $\mathbb{Q}$.
\end{lem}

\begin{proof}  Let $P$ be one of the two  polynomial types. 
Descartes rule of signs says that  
a  polynomial  with 
  real coefficients has positive real zeros equal to the number of sign changes  or is less than the number of sign changes  by an even number.  Since  $P$  has only one sign change we may conclude it has a unique positive root which we call $r$.
  
Next we show that $r$ is the dominant root.  It is a classical result  of Cauchy, called the Cauchy bound, that a polynomial 
of the form 
$$
x^{n} +a_{n-1}x^{n-1}+...+a_{0}, 
$$
has zeros in the closed disc of radius $r >0$, where $r$ is the unique positive real solution of 
$$
x^{n}- |a_{n-1|}x^{n-1}-...-|a_{0}|
$$

Next we claim that no other root of $P$ has absolute value greater than or equal to $r$.  if  $\xi$ is  a root of $ Q_{2m}$ with $|\xi|=r$, then

$$
r^{m+1}=|\xi^{m+1}|=|2\xi^{m-1} +2\xi^{m-2}+\dots +2\xi+1|
$$
$$
\leq  2|\xi|^{m-1} +2|\xi|^{m-2}+\dots +2|\xi|+1    
$$
$$
= 2r^{m-1} +2r^{m-2}+\dots +2r+1=r^{m+1}
$$
Now the only way the  inequality is an equality  for $m \geq 2$ is if 
$\xi$ is  positive real. Hence $\xi=r$, and thus $r$ is the dominant root for $Q_{2m}$.
The same argument works for the other  polynomial family
$\{P_{2m+1} \}_{m=1}^{\infty}$.  We have established that both families have a unique positive real root that is dominant over all roots. 

$P_{2m+1}$, by  the Eisenstein criteria (\cite{B-J-N}) is irreducible  over 
$\mathbb{Q}$, hence  over the integers. 
  Now any  monic irreducible  polynomial  $f$ over 
  $\mathbb{Q}$  has no repeated roots in $\mathbb{C}$. For otherwise, the greatest common divisor  of $f$ and  its derivative $f^{\prime}$ would be  non-trivial implying that $f$ is reducible over $\mathbb{Q}$. Thus 
  $P_{2m+1}$ is irreducible and has distinct roots.

We next consider the polynomial $Q_{2m}$. Here  the Eisenstein  Criterion does not apply.  Instead, we note  that the  characteristic polynomial, $Q_{2m}(z)$,  in this case can be factored as 
$$
(x+1)(x^{m} - x^{m-1} - x^{m-2}- \dots -z -1)=(x+1)S(x). 
$$
Focusing  on the second factor,   we can  rule out $-1$ as a root of the second factor since   upon substitution, 
\begin{equation}
S(-1)=
 \begin{cases}
   \text{    } 1,   & \text{if m even} \\
    -2,  & \text{if m odd}
\end{cases}
\end{equation}

Thus we see that $S(x)$  does not share a root with  the factor $(x+1)$, and  moreover, it has  simple roots  (see \cite{Wol}).  We can now  conclude that $Q_{2m}$  has distinct  roots.  
\end{proof}

Denote the positive real root  of $P_{2m+1}$ by 
$\alpha_{2m+1}$, and the positive real root of $Q_{2m}$ by 
 $\beta_{2m}$. 
 Thus we have the odd sequence of dominant roots 
 $\{\alpha_{2m+1}\}_{m=1}^{\infty}$ and the even sequence of dominant roots 
$\{\beta_{2m}\}_{m=2}^{\infty}$.

\begin{lem}[Properties of the dominant roots] Let $\alpha_{2m+1}$ be the dominant root of $P_{2m+1}$, and $\beta_{2m}$ the dominant root of $Q_{2m}$. 
\begin{enumerate}
\item The dominant roots $\{\alpha_{2m+1}\}_{m=1}^{\infty}$ of the polynomials $\{P_{2m+1}\}$ form an increasing sequence and satisfy
 $\sqrt{2} \leq  \alpha_{2m+1} \leq 2$, where 
 $\alpha_{3}=\sqrt{2}$, and $\alpha_{\infty}=2$
 \item The dominant roots $\{\beta_{2m}\}_{m=2}^{\infty}$ of the polynomials $\{Q_{2m}\}$
form an increasing sequence  and satisfy 
 $\sqrt{2} < \beta_{2m} <2$
\item $\beta_{2m} < \alpha_{2m+1}$ for $m=2,3,4,...$.
\end{enumerate}
\end{lem}

\begin{proof}
Note that $P_{2m+1}(\sqrt{2})<0$ and $P_{2m+1}(2) >0$ coupled with the fact that the positive root $\alpha_{2m+1}$ is unique imply 
$\sqrt{2}<\alpha_{2m+1} <2$. Now, a straightforward calculation shows that $P_{2m+3}<P_{2m+1}$ for 
$0<x<2$, and thus the roots 
$\{\alpha_{2m+1}\}_{m=1}^{\infty}$
are increasing.  This proves item (1). 

For item (2), a computation shows 
$Q_{2m}(\sqrt{2})<0$ and $Q_{2m}(2)=3$. So,
$\sqrt{2} <\beta_{2m} < 2$. Also, a short  computation yields
$Q_{2m+2} <Q_{2m}$ and thus 
$\{\beta_{2m}\}_{m=2}^{\infty}$ is an increasing sequence. 

Finally, item (3)  follows from 
$P_{2m+1}(x)-Q_{2m}(x)=-1<0$, and so 
$P_{2m+1}(x)<Q_{2m}(x)$ and therefore  the unique positive zero of 
$Q_{2m}$   must occur before the unique positive zero of $P_{2m+1}$.
\end{proof}

\section{The linear system associated to a  recurrence relation (solving for coefficients)}\label{sec: linear system}

Let $P$ be one of the polynomials $P_{2m+1}$ or $Q_{2m}$.
Then $P$ has $m+1$ distinct roots, and a  dominant positive root  (Lemma \ref{lem: dominant root}).  Denote the roots by $\lambda_{0},...,\lambda_{m}$ and designate $\lambda_{0}$ to be  the dominant positive root. A general solution for the counting (linear recurrence)  problem has the form
\begin{displaymath}
g(t)=a_{0}\lambda_{0}^{t}+...+a_m \lambda_{m}^{t}
\end{displaymath}
Now assume the initial conditions are 
$$
g(0)=b_{0},  g(1)=b_{1},..., g(m)=b_{m}
$$
we obtain  the $(m+1) \times (m+1)$ linear system 
\begin{equation}
\begin{pmatrix}
     1 & 1 & \cdots & 1  \\
   \lambda_{0} & \lambda_{1} & \cdots & \lambda_{m}  \\
  \lambda_{0}^{2} & \lambda_{1}^{2} & \cdots & \lambda_{m}^{2}  \\
\vdots& \vdots & \cdots & \vdots \\
  \lambda_{0}^{m} & \lambda_{1}^{m} & \cdots & \lambda_{m}^{m}  \\
\end{pmatrix}
\begin{pmatrix}
      a_{0}   \\
      a_{1}   \\
      \vdots     \\
       a_{m}
\end{pmatrix}
=
\begin{pmatrix}
    b_{0}   \\
    b_{1}     \\
    \vdots     \\
    b_{m}    \\
\end{pmatrix}
\end{equation}
Denoting the unique solution as $(a_{0},,,,a_{m})$, we see that 
the constant in the asymptotic solution of the counting problem is the coefficient $a_{0}$. 

\begin{rem} The initial conditions need not start at $t=0$. In that case the above linear system should  be adjusted appropriately. 
\end{rem}

\noindent{\bf Example:} Let $k=4$. In this illustrative example we determine the growth rate and its leading coefficient  $d_4$ for the
number of reciprocal geodesics based at the order two cone point of the Hecke surface $X_{4}$. That is, we work with the $(2,4,\infty)$ Hecke group.  Noting that $m=2$, the characteristic polynomial associated to the linear recurrence relation is
\begin{displaymath}
x^{3}-2x-1=
 \left(x-\frac{1+\sqrt{5}}{2}\right)
 \left(x-\frac{1-\sqrt{5}}{2}\right)
\left( x+1\right)
\end{displaymath}

Thus the general solution of this recursion problem is 
\begin{displaymath}
\mathcal{N}_{2t}=d_{4}\left(\alpha\right)^{t}
+a\left(\bar{\alpha}\right)^{t}
+b\left(-1\right)^{t}
\end{displaymath}
where we have set  \( \alpha = \frac{1 + \sqrt{5}}{2} \), the golden ratio,  \( \bar{\alpha} = \frac{1 - \sqrt{5}}{2} \), its conjugate, and  $d_{4}, a$, and $b$ are constants to   be determined by the initial conditions.  Note that  $\frac{1+\sqrt{5}}{2}$ is the dominant root. It is approximately 
$1.61803$. Now for $m=2$, by  Proposition \ref{thm:recurrence relation even k},   we can read off the initial conditions  as 
$\mathcal{N}_{2(2)}=2, \mathcal{N}_{2(3)}=1$,
and $\mathcal{N}_{2(4)}=4$. This leads to the following linear system with three  equations and three  unknowns $d_{4}, a, b$. 

\[
A
\begin{bmatrix}
d_4 \\
a \\
b
\end{bmatrix}
=
\begin{bmatrix}
2 \\
1\\
4
\end{bmatrix}
\]

where

\[
A =
\begin{bmatrix}
\alpha^2 & \bar{\alpha}^2 & 1 \\
\alpha^3 & \bar{\alpha}^3 & -1 \\
\alpha^4 & \bar{\alpha}^4 & 1
\end{bmatrix}
\]

Numerically approximating $A$  and computing its inverse we get 

\[
A^{-1} =
\begin{bmatrix}
0.04033 & 0.10557 & 0.06525 \\
4.95967 & 1.89443 & -3.06525 \\
-1.00000 & -1.00000 & 1.00000
\end{bmatrix}
\]
and thus

\[
\begin{bmatrix}
d_4 \\
a \\
b
\end{bmatrix}
=
A^{-1}
\begin{bmatrix}
2 \\
1 \\
4
\end{bmatrix}
=
\begin{bmatrix}
\phantom{-}0.44721 \\
-0.44721 \\
\phantom{-}1.00000
\end{bmatrix}
\]

We conclude that  $d_{4}=0.44721$ as in 
Table \ref{ta: poly}.

\end{document}